\documentclass[11pt]{article}
\usepackage{amsmath, amssymb, amsfonts, amsthm, graphicx, relsize}

\usepackage[labelsep=period, font=footnotesize, labelfont=bf]{caption}

\usepackage{cmap}

\usepackage[T1]{fontenc}

\newcommand\undermat[2]{%
\makebox[0pt][l]{$\smash{\underbrace{\phantom{%
\begin{matrix}#2\end{matrix}}}_{\text{$#1$}}}$}#2}

\newtheorem{theorem}{Theorem}
\newtheorem{lem}[theorem]{Lemma}

\newtheorem{cor}[theorem]{Corollary}

\newtheorem{thm}[theorem]{Theorem}

\newtheorem{defi}[theorem]{Definition}
\newtheorem{exm}[theorem]{Example}

\newcommand{\p}{\mathlarger{\mathlarger{\mathlarger{+}}}}
\newcommand{\n}{\mathlarger{\mathlarger{\mathlarger{-}}}}

\title{\Large\bf  Hadamard matrices with few distinct types\vspace{1cm}}

\author{
{\Large A. Mohammadian} \, and \, {\Large B. Tayfeh-Rezaie}\\[5mm]
School of Mathematics,\\
Institute for Research in Fundamental Sciences\,(IPM),\\
P.O. Box 19395-5746, Tehran, Iran\\[2mm]
$\text{\itshape{\textsf{ali$\mathsmaller{\_}$m@ipm.ir}}}$ \quad and \quad
$\text{\itshape{\textsf{tayfeh-r@ipm.ir}}}$\vspace{1cm}}

\date{}

\begin{document}
\maketitle

\begin{abstract}
\noindent The notion of type of quadruples of rows   is proven to be useful in the classif{}ication of   Hadamard matrices.  In this  paper, we  investigate Hadamard matrices with few distinct types. Among other results,      the Sylvester Hadamard matrices  are shown to be  characterized    by  their  spectrum of   types.

\vspace{5mm}
\noindent {\bf Keywords:}  Hadamard matrix, Prof{}ile, Sylvester Hadamard matrix, Type. \\[1mm]
\noindent {\bf AMS Mathematics Subject Classif{}ication\,(2010):}   05B20, 15B34.
\end{abstract}

\vspace{5mm}

\section{Introduction}

A {\sl  Hadamard matrix} of order $n$ is an $n\times n$  matrix $H$ with entries in $\{-1, 1\}$   such that $HH^\top=nI$, where $H^\top$ is
the transpose of $H$ and $I$ is the $n\times n$ identity matrix.
It is well known that the order of a Hadamard matrix is $1$, $2$,   or a multiple of $4$  \cite{pal}.
For a very long time until now,  it is unknown that Hadamard matrices of order $n$ exist for any  $n$   divisible by $4$.
The order $668$ is the  smallest  for which the existence of a Hadamard matrix is  open to question  \cite{kha}.
Hadamard matrices were f{}irst investigated by    Sylvester  in  \cite{syl} who gave  an explicit construction for  Hadamard matrices of  any order which is a power of $2$.
Such matrices were later  considered by    Hadamard  as solutions to the problem of f{}inding the maximum determinant of an $n\times n$ matrix  with entries from the complex unit disk  \cite{had}.
Since then,  Hadamard matrices have been widely studied and haven found many applications  in combinatorics and other    scientif{}ic areas \cite{hor}.

Two Hadamard matrices are said to be {\sl equivalent} if one can be obtained from the other by a sequence of row negations, row permutations, column negations, and column permutations.
Classif{}ication of Hadamard matrices up to order $32$, with respect to the   equivalence  relation,  has been fulf{}illed by several authors.
For references we refer to  \cite{khar}.
The resulting classif{}ication is shown   in Table \ref{num}.
As it can be seen from Table \ref{num},   a combinatorial explosion in the number of Hadamard matrices occurs  in   the order $32$.
Full classif{}ication in  order $36$ or  more  seems to be   dif{}f{}icult and perhaps   inaccessible.
\begin{table}[h]
\begin{center}
\begin{tabular}{|c|c|c|c|c|c|c|c|c|c|c|}\hline
$n$ & 1 & 2 & 4 & 8 & 12 & 16 & 20 & 24 & 28 & 32  \\ \hline
\texttt{\#} & 1 & 1 & 1 & 1 & 1 & 5 & 3 & 60 & 487 & 13710027  \\ \hline
\end{tabular}
\caption{The number of equivalence classes of Hadamard matrices of order $n\leqslant32$.}\label{num}
\end{center}
\end{table}

In the above mentioned classif{}ications, the authors associated an integer number, called  type,  to any quadruple of the rows of a Hadamard matrix.
We give the  def{}inition of type in the next section.
It   seems  that the notion of type deserves to be    investigated to a greater  extent.
Apparently,  Hadamard matrices with  few distinct types are very rare and  have  nice combinatorial  properties.
For instance, the  Sylvester Hadamard matrices have only two distinct types for quadruples of rows.
Furthermore, there are f{}ive Hadamard matrices obtained from strongly regular graphs  on $36$ vertices  with exactly   two distinct types \cite{spe}.
In this  paper, we show that there exists  no Hadamard matrix of order larger than $12$ whose  quadruples of rows are all  of the same type.
We then focus on Hadamard metrics with   two distinct types.
Among other results, it is established   that   the   Sylvester Hadamard matrices  are   characterized    by  their  spectrum of   types.

\section{Preliminaries}

In this section, we f{}ix our  notation and present some preliminary results.  We  denote  the zero    vector and the all one    vector of   length  $k$  by  $\mathsf{0}_k$ and $\mathsf{1}_k$,  respectively.  A zero matrix is denoted by ${\bf 0}$. For convenience,   we  respectively use the notation $${\buildrel \mathlarger{\, \,  r} \over \p} \, \, \text{ and  } \, \,  {\buildrel \mathlarger{s} \over \n}$$
instead of
$$\underbrace{1 \, \cdots \, 1}_r \, \,  \text{ and } \, \,  \underbrace{-\!1 \, \cdots \, -\!\!1}_s\cdot$$
We   drop the superscripts  whenever  there is no danger of confusion.

Let $H$ be a Hadamard matrix of order $n$.
We know  from \cite{coo} that,  by a sequence of  row  negations,    column negations,  and   column   permutations, every four  distinct  rows $i, j, k, \ell$
of $H$ may be transformed uniquely to the   form
$$\begin{array}{rrrrrrrrrr} &  & s &  t  & t  & s & t  & s   & s  & t\\
i & :&  \p & \p &
\p &\p &
\p &\p    &
\p &\p\\
j & :& \p& \p& \p& \p& \n & \n & \n & \n\\
k  & :& \p& \p& \n & \n & \p& \p& \n & \n\\
\ell & :& \p& \n & \p& \n & \p& \n & \p& \n
\end{array}$$
for some $s, t$ with $s+t=n/4$ and  $0\leqslant t\leqslant\lfloor n/8\rfloor$.
Following \cite{kim}, we def{}ine   the {\sl type} of the four rows   $i, j, k, \ell$ as  $T_{ijk\ell}=t$. It is straightforward to   check  that   $T_{ijk\ell}=\tfrac{n-P_{ijk\ell}}{8}$, where $$P_{ijk\ell}=\left|\sum_{r=1}^n h_{ir}h_{jr}h_{kr}h_{\ell r}\right|$$
assuming  that $h_{uv}$ is   the $(u, v)$-entry of $H$.
This in particular  shows that `type' is an equivalence invariant, meaning that   any permutation or negation of rows  and columns leaves the type unchanged.

The following lemma plays  a key role in the sequel of  paper.

\begin{lem}\label{type}{\sl
Let $H$ be a  Hadamard matrix of order $4m$.
Fix  three  rows of $H$  and let $\kappa_t$ be the number of other  rows   which are of type $t$ with theses   three rows.
Then
$$\sum_{t=0}^{\left\lfloor\frac{m}{2}\right\rfloor}\kappa_t(m-2t)^2=m^2.$$
}\end{lem}

\begin{proof}
Let $n=4m$.
Without loss  of generality, assume that the f{}ixed three rows of $H$ in  the   form
$$\begin{array}{rrrr}
m &m &m &m\\
\p &\p &\p &\p\\
\p & \p & \n & \n\\
\p & \n & \p & \n\\
\end{array}$$
and have been put as  the f{}irst  three  rows of $H$.
Let     $x^\top=(\mathsf{1}_m,  \mathsf{0}_m, \mathsf{0}_m, \mathsf{0}_m)$.   By the def{}inition of type, we deduce that $(Hx)^\top$ is of the form
$$\Big(m, m, m, \pm\big(m-2T_{1234}\big),  \pm\big(m-2T_{1235}\big), \ldots, \pm\big(m-2T_{123n}\big)\Big).$$
Since  $(Hx)^\top(Hx)=x^\top H^\top Hx=nx^\top x=nm$, we obtain that
$$nm=(Hx)^\top(Hx)=3m^2+\sum_{t=0}^{\left\lfloor\frac{m}{2}\right\rfloor}\kappa_t(m-2t)^2,$$ as desired.
\end{proof}

\begin{defi}{\rm
Let $H$ be a  Hadamard matrix of order $n$. For any   triple $\{i, j, k\}$ of the rows of $H$,  denote by  $\kappa_t$  the  number of    rows $\ell\notin\{i, j, k\}$  of $H$     with $T_{ijk\ell}=t$. Let  $\kappa_{t_1}, \ldots, \kappa_{t_r}$ be  the distinct non-zero elements in  $\{\kappa_t \, | \, 0\leqslant t\leqslant n/8\}$. We def{}ine  the   {\sl type}  of $\{i, j, k\}$ as
$$\left(\begin{array}{ccc}   t_1 & \ldots & t_r \\ \kappa_{t_1} &  \ldots   & \kappa_{t_r}\end{array}\right).$$
Also, we def{}ine    the   {\sl prof{}ile}  of $H$ to be  the multiset of the types  of all triples of the rows of $H$.
}\end{defi}

The  prof{}ile   of  Hadamard matrices  can be used in studying Hadamard equivalence, since  two equivalent Hadamard matrices have the same prof{}ile, however,    the inverse  is not true in general, as we will see  in Example \ref{exa13}.

The following result   originally proven in Proposition 2.1 of   \cite{kim} is an easy   consequence   of  Lemma \ref{type}.

\begin{cor}{\sl
Let $n\geqslant8$ and  $H$ be a  Hadamard matrix of order $n$. If there exists a  quadruple $\{i, j, k, \ell\}$ of rows  of  $H$ with $T_{ijk\ell}=0$, then $n\equiv0 \, (\mathrm{mod} \, 8)$.
}\end{cor}

The following result  is a generalization of Lemma 2 of   \cite{khara}.

\begin{cor}\label{1new}{\sl
Let $n\geqslant4$ and  $H$ be a  Hadamard matrix of order $n$.
If  there  exist  three  distinct  rows $i, j, k$ of $H$ such that all quadruples   $\{i, j, k, \ell\}$  of rows     are  of the same type, then $n=4$ or $n=12$.
}\end{cor}

\begin{proof}
Let $n=4m$. Assume that  for  three  distinct  rows $i, j, k$ of $H$, four rows    $i, j, k, \ell$    are   of type $t$  for any $\ell\notin\{i, j, k\}$.
By Lemma  \ref{type}, we have $(n-3)(m-2t)^2=m^2$. This means that $n^2$ is divisible by $n-3$. Therefore,  $9=n^2-(n^2-9)$ is divisible by $n-3$ and  we conclude  that $n=4$ or $n=12$.
\end{proof}

\begin{cor}\label{1bad}{\sl
Let $n\geqslant4$ and  $H$ be a  Hadamard matrix of order $n$. If all quadruples of rows are of the same type, then $n=4$ or $n=12$.
}\end{cor}

\section{Hadamard matrices with two  distinct types}

In this section, we investigate Hadamard matrices whose types of quadruples  of   rows take   few   distinct  values. By Corollary \ref{1bad}, any Hadamard matrix of order larger than $12$ has at least two distinct types. Thus, it is   natural to  ask about    Hadamard matrices with exactly two  distinct types.  We expect such  matrices to be  very rare and    structurally   nice. The complete classif{}ication of theses objects  seems   dif{}f{}icult.   We here obtain   some partial results. In particular, we examine the Hadamard matrices of order $n$ having  types $\alpha$ and $\beta$ for any quadruple of rows with   $(\alpha, \beta)\in\{(0, \tfrac{n}{8}), (1,  \tfrac{n-4}{8}), (\tfrac{n}{16},  \tfrac{n}{8})\}$. Note that theses pairs of types satisfy  the equation given in Lemma \ref{type}.

The following lemma is useful in eliminating  some possible  solutions of the equation stated  in Lemma \ref{type}.

\begin{lem}\label{ij}{\sl
Let $i, j, k, p, q$ be f{}ive distinct  rows of a Hadamard matrix of order $4m$. Then $T_{ijkp}+T_{ijkq}\geqslant m/2$.
Moreover, if  the equality occurs, then  these f{}ive  rows can be written as
\begin{equation}\label{=}\begin{array}{rrrrrrrrrrrrrrrr}
&   &  \frac{m}{2} &  t'  &  s'  &  t  & s   & \frac{m}{2}   &  t  & s  &  \frac{m}{2}  &   \frac{m}{2}  &  t'  & s' \\
i & : & \p &  \p  & \p    &    \p  & \p  & \p     & \p  & \p   & \p &  \p  & \p   & \p  \\
j & : & \p &  \p  & \p   &  \p  & \p  & \p   & \n &  \n  & \n    &  \n & \n & \n   \\
k & : & \p &  \p  & \p       &  \n & \n & \n     &  \p  & \p  & \p   & \n &  \n  & \n     \\
p & : & \p &  \p  &\n      &  \p  & \n  & \n     &  \p  & \n  & \n   & \p &  \p  & \n    \\
q & : & \p &  \n  & \p   &  \n  & \p  & \n    &  \n  & \p  &\n   & \p &  \n  & \p
\end{array}\end{equation}
where $t=m/2-t'=T_{ijkp}$ and $s=m/2-s'=T_{ijkq}$.
}\end{lem}

\begin{proof}
Without loss of generality, we may assume that
$$\begin{array}{llllllllllllllllllll}   &   &  a_1 &  a_2 & a_3  & a_4 & b_1  &   b_2 & b_3 & b_4 & c_1  &   c_2  & c_3  & c_4 &  d_1 &  d_2 & d_3  & d_4 \\
i & : & \p &  \p  & \p  & \p  &  \p &  \p  & \p  & \p   & \p &  \p  & \p  & \p   & \p &  \p  & \p  & \p   \\
j & : & \p &  \p  & \p  & \p  &  \p &  \p  & \p  & \p   & \n &  \n  & \n  & \n   & \n &  \n & \n & \n   \\
k & : & \p &  \p  & \p  & \p  &  \n &  \n & \n & \n   & \p &  \p  & \p  & \p   & \n &  \n  & \n  & \n   \\
p & : & \p &  \p  &\n  & \n  &  \p &  \p  & \n  & \n   & \p &  \p  & \n  & \n   & \p &  \p  & \n  & \n  \\
q & : & \p &  \n  & \p  & \n &  \p &  \n  & \p  & \n   & \p &  \n  & \p  &\n   & \p &  \n  & \p  & \n\cdot
\end{array}$$
By the def{}inition of type and since $T_{ijkp}=t$ and $T_{ijkq}=s$, we have
\begin{eqnarray*}
\begin{array}{ll}
a_3+a_4=b_1+b_2=c_1+c_2=d_3+d_4=t,
\\ \vspace{-2mm} \\
a_2+a_4=b_1+b_3=c_1+c_3=d_2+d_4=s,
\end{array}
\end{eqnarray*}
and
$$a_1+a_2+a_3+a_4=b_1+b_2+b_3+b_4=c_1+c_2+c_3+c_4=d_1+d_2+d_3+d_4=m.$$
Solving  the equations   above, we obtain that
\begin{equation}\label{para}
\begin{array}{lll}
\left\{\begin{array}{ll} a_2=m-t-a_1  \\ a_3=m-s-a_1  \\ a_4=t+s-m+a_1, \end{array}\right.  & \hspace{1cm}  &
\left\{\begin{array}{ll} b_2=t-b_1\\  b_3=s-b_1 \\  b_4=m-t-s+b_1, \end{array}\right. \\  & \hspace{1cm}  & \\
\left\{\begin{array}{ll} c_2=t-c_1 \\   c_3=s-c_1 \\    c_4=m-t-s+c_1, \end{array}\right.  & \hspace{1cm}  &
\left\{\begin{array}{ll}  d_2=m-t-d_1 \\   d_3=m-s-d_1 \\  d_4=t+s-m+d_1. \end{array}\right.
\end{array}
\end{equation}
The inner product of two rows $p$ and $q$ is equal to $4(a_1+b_1+c_1+d_1-m)$. So  the orthogonality of   rows $p$ and $q$ implies that  $a_1+b_1+c_1+d_1=m$.
Since $a_4, d_4\geqslant0$, we deduce that both  $a_1$ and $d_1$ are at least $m-t-s$. Therefore, $m\geqslant a_1+d_1\geqslant2(m-t-s)$ and so   $t+s\geqslant m/2$, as desired.

If $t+s=m/2$, then $a_1+d_1=m$. As mentioned above,   since $a_1$ and $d_1$ are at least $m-t-s$, we conclude that  $a_1=d_1=m/2$. By $a_1+b_1+c_1+d_1=m$, we f{}ind  that $b_1=c_1=0$. Now, the result follows from \eqref{para}.
\end{proof}

\begin{thm}\label{16t}{\sl
There exists no    Hadamard matrix of order $16t$ whose  all quadruples of rows are of type $t$ or $2t$.
}\end{thm}

\begin{proof}
By contradiction, assume that there exists a  Hadamard matrix $H$  of order $n=16t$ whose  all quadruples of rows are of type $t$ or $2t$.
Let $\kappa_t$ and $\kappa_{2t}$ be the number of  rows which respectively are of type $t$ and $2t$  with the f{}irst    three rows.
By applying Lemma \ref{type}, we f{}ind that $\kappa_t=4$ and  $\kappa_{2t}=n-7$.
Without loss of generality, we may assume that $T_{1234}=T_{1235}=T_{1236}=T_{1237}=t$.
For any pair   $p, q\in\{4, 5, 6, 7\}$, since the equality holds  in   Lemma \ref{ij},   f{}ive rows $1, 2, 3, p, q$ can be written as  \eqref{=}. Thus, it is straightforward to check that  we   necessarily have   the following conf{}iguration:
$$\begin{array}{llllllllllllllllllll}   &   &  t &  t  &  t  &  t  & t   & t   & t  & t   &  t  &  t  &  t  & t&  t  &t  &t  &t   \\
1 & : & \p& \p &  \p  & \p    &    \p  & \p & \p & \p     & \p  & \p   & \p &  \p   & \p &  \p & \p   & \p  \\
2 & : & \p& \p &  \p  & \p   &  \p  & \p & \p & \p   & \n &  \n  & \n    &  \n & \n    &  \n & \n & \n   \\
3 & : & \p& \p &  \p  & \p       &  \n & \n & \n & \n     &  \p  & \p  & \p    & \p  & \n & \n   & \n &  \n       \\
4 & : & \p& \p &  \p  &\n      &  \p  & \n  & \n & \n     &  \p  & \n  & \n & \n & \p& \p & \p   & \n    \\
5 & : & \p& \p &  \n  & \p   &  \n  & \p  & \n & \n    &  \n  & \p  &\n &\n  & \p & \p &  \n  & \p \\
6 & : & \p& \n &  \p  & \p   &  \n  & \n  & \p & \n    &  \n  & \n  &\p &\n  & \p & \n &  \p  & \p \\
7 & : & \n & \p &  \p  & \p   &  \n  & \n  & \n& \p    &  \n  & \n  &\n  &\p  & \n & \p &  \p  & \p\cdot
\end{array}$$
It turns  out that $P_{4567}=n$ and so   $T_{4567}=0$, a contradiction.
\end{proof}

It  has been  shown in \cite{hal}  that there are exactly f{}ive equivalence classes of Hadamard matrices of order $16$. We   prove the following result without any reference to these      equivalence classes.

\begin{cor}{\sl
Every     Hadamard matrix of order $16$ has  four rows of  type $0$.
}\end{cor}

\begin{proof}
Lemma \ref{type} yields that  each  triple of rows of a  Hadamard matrix of order $16$ is of type
$$\left(\begin{array}{cc}   0  & 2 \\  1    & 12 \end{array}\right) \quad \text{ or } \quad  \left(\begin{array}{cc}   1  & 2 \\  4    & 9 \end{array}\right).$$
Now,    the result follows from Theorem \ref{16t}.
\end{proof}

Recall   that the {\sl Hadamard   product} of two  $(-1, 1)$-vectors $a=(a_1, \ldots, a_n)$ and $b=(b_1, \ldots, b_n)$ is def{}ined    as   $a\circ b=(a_1b_1, \ldots, a_nb_n)$.
We also  def{}ine $\sigma(a)=|a_1+\cdots+a_n|$. It is not hard to check  that
\begin{equation}\label{norm}
\sigma(a\circ b)\geqslant\sigma(a)+\sigma(b)-n.
\end{equation}

Roughly specking, the following theorem states that there is no     large gap   between the  types of  quadruples of   rows of a Hadamard matrix whose order is not a power of $2$.

\begin{thm}\label{full}{\sl
Let $H$ be a Hadamard matrix   of order $n$ and let $r<n/16$.   Suppose that for every three distinct  rows $i, j, k$ of $H$,  there exists   a   row  $\ell$ with    $T_{ijk\ell}\leqslant r$    and   no row $x$ with   $r<T_{ijkx}\leqslant2r$. Then $n$ must be  a power of $2$.
}\end{thm}

\begin{proof} By Lemma \ref{ij}, for every three distinct  rows $i, j, k$ of $H$,  there exists   a  unique   row  $\ell$ with    $T_{ijk\ell}\leqslant r$. We  say a set $\mathcal{S}$   of rows  of $H$  to be `full'  if for  every distinct  rows   $i, j, k\in\mathcal{S}$,  the   unique   row  $\ell$ with    $T_{ijk\ell}\leqslant r$ is   contained in $\mathcal{S}$. Trivially, $H$ has      a full  set of size $4$. We claim   that any  full  set of size $s<n$  can be extended to   a full  set of size $2s$. Clearly, the claim  concludes  the assertion of the theorem.

Suppose that   $\mathcal{S}=\{a_1, \ldots, a_s\}$   is a full set in $H$. Choose   an arbitrary  row $b_1$ in $H$  outside of $\mathcal{S}$ and,   for $i=2, \ldots, s$, let $b_i$  be the unique  row in $H$ such that     $T_{a_1a_ib_1b_i}\leqslant r$. For any $i\geqslant2$, we may write $b_i=a_1\circ a_i\circ b_1\circ \beta_i$  for a suitable $(-1, 1)$-vector $\beta_i$. Note that $\sigma(\beta_i)=\sigma(a_1\circ a_i\circ b_1\circ b_i)=n-8T_{a_1a_ib_1b_i}\geqslant n-8r$.
Since $\mathcal{S}$ is a full set and $b_1$ is not in $\mathcal{S}$, so are  $b_2,  \ldots, b_s$. If $b_i=b_j$, then  $T_{a_1b_1b_i\ell}\leqslant r$ for $\ell=a_i$ and $\ell=a_j$, a contradiction.  So, $\mathcal{S}'=\mathcal{S}\cup\{b_1, \ldots, b_s\}$ is of size $2s$. Now, we prove that $\mathcal{S}'$ is full.   It clearly suf{}f{}ices  to consider only  the  following two cases:

\noindent{\bf{\textsf{Case 1.}}} For every  $2\leqslant i<j\leqslant s$,    we  show     $T_{a_ia_jb_ib_j}\leqslant r$ and $\sigma(\beta_i\circ\beta_j)\geqslant n-8r$.

From \eqref{norm}, we have
\begin{align}\label{bet}
n-8T_{a_ia_jb_ib_j}&=\sigma(a_i\circ a_j\circ  b_i\circ  b_j)=\sigma(\beta_i\circ\beta_j)\\\nonumber&\geqslant\sigma(\beta_i)+\sigma(\beta_j)-n\geqslant n-16r
\end{align}
and so $T_{a_ia_jb_ib_j}\leqslant2r$. By the assumption   of the theorem, we have   $T_{a_ia_jb_ib_j}\leqslant r$.  The second inequality   follows from \eqref{bet}.

\noindent{\bf{\textsf{Case  2.}}} For any  quadruple $\{a_i, a_j, a_k, a_\ell\}$  of type at most $r$, we  show   $T_{a_ia_jb_kb_\ell}\leqslant r$  and $T_{b_ib_jb_kb_\ell}\leqslant r$.

From \eqref{norm}, we have
\begin{align}\label{sbet}
n-8T_{a_ia_jb_kb_\ell}&=\sigma(a_i\circ a_j\circ  b_k\circ  b_\ell)=\sigma(a_i\circ a_j\circ  a_k\circ  a_\ell\circ\beta_i\circ\beta_j)\\\nonumber&\geqslant\sigma(a_i\circ a_j\circ  a_k\circ  a_\ell)+\sigma(\beta_i\circ\beta_j)-n\geqslant n-16r
\end{align}
and so $T_{a_ia_jb_kb_\ell}\leqslant2r$. The assumption   of the theorem  results in  $T_{a_ia_jb_kb_\ell}\leqslant r$. By \eqref{sbet}, we obtain that $\sigma(a_i\circ a_j\circ  a_k\circ  a_\ell\circ\beta_i\circ\beta_j)\geqslant n-8r$ for any  quadruple $\{a_i, a_j, a_k, a_\ell\}$  of type at most $r$. This  along with  the second inequality in Case 1 give
\begin{align*}
n-8T_{b_ib_jb_kb_\ell}&=\sigma(b_i\circ b_j\circ  b_k\circ  b_\ell)=\sigma(a_i\circ a_j\circ  a_k\circ  a_\ell\circ\beta_i\circ\beta_j\circ\beta_k\circ\beta_\ell)\\&\geqslant\sigma(a_i\circ a_j\circ  a_k\circ  a_\ell\circ\beta_i\circ\beta_j)+\sigma(\beta_k\circ\beta_\ell)-n\geqslant n-16r,
\end{align*}
implying $T_{b_ib_jb_kb_\ell}\leqslant2r$. By the assumption   of the theorem,     $T_{b_ib_jb_kb_\ell}\leqslant r$  which completes  the proof.
\end{proof}

The following consequence  immediately follows from Theorem \ref{full}.

\begin{cor}\label{4h}{\sl
Let $H$ be a Hadamard matrix   of order $n$ such that  for every three distinct  rows $i, j, k$ of $H$,  there exists a    row $\ell$   with    $T_{ijk\ell}<n/24$.  Then $n$ is  a power of $2$.
}\end{cor}

Consider a Hadamard matrix  $H$ of order    $n>12$. Assume that $n$ is not a power of $2$ and $H$ has   exactly two distinct types  $\alpha$ and   $\beta$ for the   quadruples  of   rows with    $\alpha<\beta$.  Then   Lemma \ref{type} and Corollaries \ref{1new} and \ref{4h} result in  $$\frac{n}{24}\leqslant\alpha\leqslant\frac{n}{8}\left(1-\frac{1}{\sqrt{n-3}}\right)\leqslant\beta\leqslant\frac{n}{8}.$$

We recall that the   {\sl Sylvester Hadamard matrices}   are  recursively def{}ined as follows:
$$\mathsf{H}_1={\begin{bmatrix}1\end{bmatrix}} \quad   \text{ and } \quad
\mathsf{H}_{2^r}=\left[\begin{array}{rr} \mathsf{H}_{2^{r-1}}&  \mathsf{H}_{2^{r-1}} \\ \mathsf{H}_{2^{r-1}}& -\mathsf{H}_{2^{r-1}}\end{array}\right] \quad \text{ for } r=1, 2, \ldots$$
It follows from Theorem 4 of  \cite{coo}  that every quadruple of   rows of $\mathsf{H}_{2^r}$  is  of type $0$ or $2^{r-3}$  for all    $r\geqslant3$.  We below show that the converse    is also  true.

\begin{thm}\label{sylv}{\sl
Let $H$ be a Hadamard matrix   of order $8t$ whose  all quadruples of rows are of type $0$ or $t$. Then $H$ is equivalent  to  the   Sylvester Hadamard matrix.
}\end{thm}

\begin{proof}
Fix three rows   of $H$ and let $\kappa_0$ and $\kappa_t$ be the number of other  rows which respectively are of type $0$ and $t$  with the these f{}ixed   rows.
By applying Lemma \ref{type}, we f{}ind that $\kappa_0=1$ and    $\kappa_t=n-4$, where $n=8t$.  It is easy   to see that,  for every triple $\{i, j, k\}$ of rows
of $H$,  the vector $i\circ j\circ k$ is equal to the unique  row $\ell$ in $H$ with  $T_{ijk\ell}=0$
up to  negation.
This means that if we write the f{}irst three rows of $H$ as the form
$$\begin{array}{rrrr}
2t & 2t & 2t & 2t \\
\p &\p &\p &\p\\
\p & \p & \n & \n\\
\p & \n & \p & \n\\
\end{array}$$
then we may consider
$$\begin{array}{rrrr}
2t  & 2t & 2t & 2t \\
\p &\n &\n &\p
\end{array}$$
as   the forth row of $H$.
By a sequence of   column permutations,   we may consider  the  $4\times n$    top submatrix of $H$  as
$$\left[\begin{array}{c|c|c} \mathsf{H}_4 & \cdots & \mathsf{H}_4 \end{array}\right].$$
In order to proceed,    assume that $n$ is divisible by $2^r$, for some  $r\geqslant2$,  and the  $2^r\times n$    top submatrix of $H$ is written as
$$\left[\begin{array}{c|c|c} \mathsf{H}_{2^r} & \cdots & \mathsf{H}_{2^r} \end{array}\right].$$
Again, by a sequence of   column permutations,  we may consider  the  $2^r\times n$    top submatrix of $H$  as
\begin{equation}\label{2r}H'=\left[\begin{array}{c|c|c} \undermat{\mathlarger{\mathlarger{\tfrac{n}{2^r}}}}{K_1 \cdots K_1} &    \cdots  &
\undermat{\mathlarger{\mathlarger{\tfrac{n}{2^r}}}}{K_{2^r} \cdots K_{2^r}} \end{array}\right],\vspace{6mm}\end{equation}
where $K=\mathsf{H}_{2^r}$   and $K_i$ is the $i$th column of $K$    for  $i=1,  \ldots, 2^r$.
Let
$$\begin{array}{rrrrrr} x & : &  x_1 &   \cdots  & x_{2^r}  \end{array}$$
be any of the  remaining  rows of $H$. In view of \eqref{2r},   by a column permutation,  we may assume that
$$\begin{array}{rrr}
& \alpha_i & \beta_i \\
x_i \, \, \, : \,  &  \p &\n
\end{array}$$
for any $i$.
Since     $H'x^\top=0$, it is not hard to see that
$$K\left[\begin{array}{c} \alpha_1-\beta_1  \\\hline \vdots \\\hline \alpha_{2^r}-\beta_{2^r}\end{array}\right]={\bf 0}.$$ As $K$ is an  invertible matrix, we conclude that   $\alpha_i=\beta_i$ for any $i$. Thus,   we may rewrite the f{}irst $2^r+1$  rows of $H$ in  the form
$$\begin{array}{cccccc} K & \cdots & K  & K & \cdots & K  \\ \p & \cdots & \p & \n & \cdots & \n\cdot \end{array}$$
For any     $i\in\{2, 3,  \ldots, 2^r\}$,  $H$ has a unique  row $\rho_i=1\circ i\circ\rho_1$ corresponding to  the rows  $1$, $i$, and $\rho_1=2^r+1$ with $T_{1i\rho_1\rho_i}=0$. So, one can easily deduce  that  the   f{}irst $2^{r+1}$  rows of $H$ have  the form
$$\begin{array}{rrrrrr} K & \cdots & K & K & \cdots & K \\ K & \cdots & K & -K & \cdots & -K.\end{array}$$
This shows in particular that $n$ is divisible by $2^{r+1}$.
Also, by a sequence of   column permutations,  we may consider  the  $2^{r+1}\times n$    top submatrix of $H$  as
$$\left[\begin{array}{c|c|c} \mathsf{H}_{2^{r+1}} & \cdots & \mathsf{H}_{2^{r+1}} \end{array}\right].$$
Now,  the assertion clearly follows by repeating  the above  process.
\end{proof}

The following result is an  analogue  of Theorem \ref{sylv} and is easily derived  from Corollary \ref{4h}.

\begin{cor}{\sl
Let $H$ be a Hadamard matrix   of order $n=8t+4$ whose  all quadruples of rows are of type $1$ or $t$. Then $n\in\{4, 12, 20\}$.
}\end{cor}

\begin{exm}\label{exa13}{\rm
There are two inequivalent  Hadamard matrices of order $32$ with the same  prof{}ile
$$\left\{\left(\begin{array}{cc}   0  & 4 \\  1    & 28  \end{array}\right)^{^{\displaystyle{[1376]}}},   \left(\begin{array}{ccc}   1  & 3 & 4  \\  1 & 7    & 21 \end{array}\right)^{^{\displaystyle{[3584]}}}\right\},$$
where the exponents indicate the multiplicities.
These    Hadamard matrices, which are obtained in   \cite{khara}, satisfy   the condition of Corollary \ref{4h}.
}\end{exm}

\section{Concluding remarks}

We showed that  Hadamard matrices   with exactly  one type for quadruples  of rows exist  only  in    orders  $ 4$ and $12$. The classif{}ication of   Hadamard matrices with exactly  two distinct values for  type of  quadruples  of rows seems to be a  hard problem.  Even, in   order $36$ the problem is already hard.   We carried  out  a non-exhaustive computer   search for    Hadamard matrices of order $36$  having type   $3$ or  $4$  for quadruples  of rows. We obtained  only f{}ive  such Hadamard  matrices which had  been already found   in \cite{spe}.    It is an interesting question  if there exists    an inf{}inite family of   Hadamard matrices with exactly  two distinct types besides  the     Sylvester Hadamard matrices.

\section*{Acknowledgments}

This research  was in part supported by grants from IPM.

{}

\end{document}